\documentclass{amsart}
\usepackage[utf8]{inputenc}
\usepackage{latexsym,amsmath,amsthm}
\usepackage{amssymb}
\usepackage[pdftex,colorlinks,backref=page,citecolor=blue]{hyperref}
\usepackage{geometry}
\usepackage{comment}
\usepackage[pdftex]{graphicx}

\usepackage{graphpap, color, paralist, pstricks}

\usepackage{setspace}
\setdisplayskipstretch{1}

\newtheorem{theorem}{Theorem}[section]
\newtheorem{proposition}[theorem]{Proposition}

\newtheorem{lemma}[theorem]{Lemma}

\newtheorem{claim}[theorem]{Claim}

\newtheorem{fact}[theorem]{Fact}

\newcommand{\lam}{\lambda}

\newcommand{\cA}{\mathcal{A} }

\newcommand{\cC}{\mathcal{C} }

\newcommand{\cF}{\mathcal{F} }

\newcommand{\cI}{\mathcal{I} }

\newcommand{\bI}{\mathbf{I}}
\newcommand{\bT}{\mathbf{T}}
\newcommand{\bX}{\mathbf{X}}
\newcommand{\bY}{\mathbf{Y}}
\newcommand{\bff}{\mathbf{f}}

\newcommand{\beq}[1]{\begin{equation}\label{#1}}
	\newcommand{\enq}[0]{\end{equation}}

\newcommand{\eps}{\epsilon}

\newcommand{\nin}[0]{\noindent}

\newcommand{\ov}[0]{\overline}
\newcommand{\sub}[0]{\subseteq}

\newcommand{\E}{\mathbb{E}}
\newcommand{\pr}{\mathbb{P}}

\newcommand{\dn}[1]{d_{N({#1})}}

\newcommand{\al}{\alpha}

\title{Note on the Number of Antichains in Generalizations of the Boolean Lattice}

\author{Jinyoung Park $^\ddagger$}
\thanks{$^\ddagger$ Department of Mathematics, Courant Institute of Mathematical Sciences, New York University. Research supported by NSF grant DMS-2153844.}
\email{jinyoungpark@nyu.edu}

\author{Michail Sarantis $^*$}
\thanks{$^*$ Department of Mathematical Sciences, Carnegie Mellon University, Pittsburgh. Research is supported in part by the NSF grant DMS-2151283 and the Onassis Foundation - Scholarship F ZP 051-1/2019-2020. }
\email{msaranti@andrew.cmu.edu}

\author{Prasad Tetali $^\dagger$}
\thanks{$^\dagger$ Department of Mathematical Sciences, Carnegie Mellon University, Pittsburgh. Research supported in part by the NSF grant DMS-2151283 and Alexander M. Knaster Professorship.}
\email{ptetali@cmu.edu}

\date{}

\begin{document}
	
	\maketitle

 \begin{abstract}
We give a short and self-contained argument that shows that, for any positive integers $t$ and $n$ with $t =O\Bigl(\frac{n}{\log n}\Bigr)$,  the number $\alpha([t]^n)$ of antichains of the poset $[t]^n$ is at most
\[\exp_2\Bigl(1+O\Bigl(\Bigl(\frac{t\log^3 n}{n}\Bigr)^{1/2}\Bigr)\Bigr)N(t,n)\,,\]
where $N(t,n)$ is the size of a largest level of $[t]^n$. This, in particular, says that if $t \ll n/\log^3 n$ as $n \rightarrow \infty$, then $\log\alpha([t]^n)=(1+o(1))N(t,n)$, giving a (partially) positive answer to a question of Moshkovitz and Shapira for $t, n$ in this range.

Particularly for $t=3$, we prove a better upper bound:
\[\log\alpha([3]^n)\le(1+4\log 3/n)N(3,n),\]
which is the best known upper bound on the number of antichains of $[3]^n$.
\end{abstract}
	
	\section{Introduction}
    An \textit{antichain} of a poset $P$ is a set of elements of $P$, any two of which are incomparable in the partial order. We denote by $\al(P)$ the number of antichains of $P$. "Dedekind's Problem" of 1897 \cite{Dedekind} asks for the number of antichains of the Boolean lattice $B_n=\{0,1\}^n$. Since any subsets of a middle layer of $B_n$ is an antichain, a trivial lower bound is $\al(B_n)\geq 2^{\binom{n}{\lfloor n/2 \rfloor}}$. 
	Kleitman \cite{Kleitman1969OnDP} and subsequently Kleitman and Markowsky \cite{Kleitman1975OnDP} proved that this trivial lower bound is optimal in the logarithmic scale, namely,
	\beq{ineq_KM} \log\al(B_n)\leq \Bigl(1+O\Bigl(\frac{\log n}{n}\Bigr)\Bigr)\binom{n}{\lfloor n/2\rfloor}.\enq
\nin (In this paper, $\log$ always means $\log_2$.) 	Asymptotics for $\alpha(B_n)$ itself was obtained by Korshunov \cite{Korshunov1981} via an extremely complicated argument. Later, Sapozhenko \cite{Sapozhenko1989} gave a simpler (yet still difficult) proof for the asymptotics.
 
	Kahn \cite{Kahn01} used entropy ideas to bound the number of independent sets of a finite, regular bipartite graph. This idea is extended in \cite{Kahn2002Entropy} to the layers of a graded poset to recover the result of Kleitman and Markowsky in \eqref{ineq_KM}. Independently, Pippenger \cite{Pippenger1999EntropyAE} also gave a slightly weaker bound $\log\al(B_n)\leq \left(1+O\left(\frac{\log^{3/2} n}{n^{1/4}}\right)\right)\binom{n}{\lfloor n/2\rfloor}$. Pippenger's work also uses entropy functions, but his approach was more akin to that of \cite{Kleitman1969OnDP}.
	
	Carroll, Cooper and Tetali \cite{Carroll2012CountingAA} considered the antichain count question for the following natural generalizations of the Boolean lattice: for an integer $t\geq 2$, let $[t]^n$ be the poset consisting of all $n$-tuples $(x_1,\dots,x_n)$ of integers in $\{0,1,\dots,t-1\}$ with the partial order $\prec$ defined by $x\prec y\Leftrightarrow x_i\leq y_i$\,, for all $1\leq i\leq n$. Following Pippenger's approach, they proved an analogous result for this generalized Boolean lattice:

	\begin{theorem}[\cite{Carroll2012CountingAA}]\label{thm:CCT} For integers $t, n$ such that $1<t<n$,
		\beq{eq.CCT} \log\al([t]^n)\leq \left(1+\frac{11t^2\log t (\log n)^{3/2}}{n^{1/4}}\right)N(t,n),\enq
        where $N(t,n)$ is the size of the middle layer(s) of $[t]^n$.
	\end{theorem}

\nin In \cite{Mattner2007MaximalPO}, it was proved that for $n \rightarrow \infty$ and every $t$,
    \beq{eq_MR} N(t,n)=t^n\sqrt{\frac{6}{\pi(t^2-1)n}}\left(1+o_n(1)\right).\enq

We note that getting the right dependency of $\alpha([t]^n)$ on each of $t$ and $n$ has an interesting connection to some Ramsey-type problems, which we will briefly discuss in Section \ref{sec.Ramsey}. Tsai \cite{Tsai2019} gave the following bound which is better than the bound in \eqref{eq.CCT} for  $t \gg \frac{n^{1/8}}{(\log n)^{3/4}}$.

\begin{theorem}[\cite{Tsai2019}] For any integers $t, n \ge 1$,
\[\log\alpha([t]^n) \le N(t,n)\log(t+1).\]    
\end{theorem}

\noindent Recently, Pohoata and Zakharov \cite{Pohoata2021} used the graph container method to treat the case of constant $t$, which yields better dependency on $n$ (but dependency unspecified on $t$).

    \begin{theorem}[\cite{Pohoata2021}] \label{thm:PZ}
        $$\log\al([t]^n)\leq \left(1+C_t\frac{\log n}{\sqrt{n}}\right)N(t,n),$$
        where $C_t>0$ is a constant depending solely on $t$.
    \end{theorem}

	\subsection{Our results} Our first result is an upper bound on  $\log\alpha([3]^n)$ with a better dependency on $n$.
	\begin{theorem}\label{thm:3^n} For any integer $n \ge 1$,
		$$\log\al([3]^n)\leq \left(1+\frac{4\log 3}{n}\right)N(3,n).$$
	\end{theorem}
 
	\noindent We remark that our proof almost identically applies to $\alpha(B_n)$ to remove the logarithmic factor in \eqref{ineq_KM}. The proof of Theorem \ref{thm:3^n} is inspired by Kahn's entropy approach in \cite{Kahn2002Entropy} with a couple of points of modifications and improvements. Kahn's result applies under quite a strict condition on degrees of the vertices in the bipartite graph induced by two consecutive layers of $B_n$. However, for any $t>2$, two consecutive layers of $[t]^n$ do not satisfy this condition.	In our work, we obtain a refined version of Kahn's inequality for the independence polynomial of a bipartite graph (Theorem \ref{thm:two_level}). We then use it to upper bound the weighted sum of antichains of a graded poset (Theorem~\ref{thm:multi_level}). The seemingly complicated inequality becomes easier to handle in the case of equal weights, and analyzing it more carefully on $[3]^n$ allows us to obtain Theorem~\ref{thm:3^n}.

    Unfortunately, the proof of Theorem~\ref{thm:3^n} does not extend to $t>3$ because of the structural difference of the bipartite graphs formed by consecutive layers of $[t]^n$. However, it is not clear whether the inequality obtained from Theorem~\ref{thm:multi_level} itself fails.
    
Our second result concerns any $t$ and $n$, whenever  $n$ is large enough compared to $t$.

 	\begin{theorem}\label{thm:t^n} There exists an absolute constant $C$ such that for integers $t, n$ such that $1\le t<n/(100\log n)$,
		\beq{disp.t^n}\log \al([t]^n) \le \left(1+C\Bigl(\frac{t\log^3 n}{n}\Bigr)^{1/2}\right)N(t,n).\enq
	\end{theorem}

\nin We note that the constant 100 in the restriction on $t$ is not important -- we can make it any constant by choosing $C$ appropriately. In the proof, for example, we will see that the choice of 100 suggests $C \approx 15$.

 The proof of Theorem \ref{thm:t^n}  adapts the beautiful entropy approach in \cite{Pippenger1999EntropyAE}: with $\bff$ a random monotone Boolean function on $[t]^n$ chosen uniformly at random, we have
\[\log\alpha([t]^n)=H(\bff)\,,\]
where $H(\cdot)$ is the binary entropy function. (See Section \ref{sec.entropy} for entropy  basics.) To bound $H(\bff)$, we define a series of auxiliary random variables that determine $\bff$, and will bound the entropy of those auxiliary random variables instead. This approach builds on that of \cite{Pippenger1999EntropyAE,Carroll2012CountingAA}, but our case analyses are done in a more careful way.

\subsection{Connection to a Ramsey-type problem}\label{sec.Ramsey}

Antichains of $[t]^n$ are closely related to another interesting combinatorial problem that we briefly discuss in this section. Following \cite{Fox2011ErdsSzekerestypeTF}, for any sequence of positive integers $j_1<j_2<\ldots<j_l$, we say that the $k$-tuples $(j_i,j_{i+1}, \ldots, j_{i+k-1})$ $(i=1,2,\ldots,l-k+1)$ form a \textit{monotone path} of length $l$. Let $M_k(n,t)$\footnote{We use $M_k(n,t)$ for the notation $N_k(q,n)$ in \cite{Fox2011ErdsSzekerestypeTF} to make it consistent with our main theorems.} be the smallest integer $M$ with the property that no matter how we color all $k$-element subset of $\{1,2,\ldots,M\}$ with $n$ colors, we can always find a monochromatic monotone path of length $t$. In this language, the celebrated results of Erd\H{o}s and Szekeres \cite{Erdos-Szekeres} can be written as $M_2(n,t)=(t-1)^n+1$ and $M_3(2,t)=\binom{2t-4}{t-2}+1$. Fox, Pach, Sudakov, and Suk \cite{Fox2011ErdsSzekerestypeTF} showed that
\[2^{(t/n)^{n-1}}\le M_3(n,t) \le 2^{t^{n-1}\log t} \quad \text{for $n \ge 2$ and $t \ge n+2$},\]
and suggested closing the gap between the lower and upper bound as an interesting question. Answering this question, Moshkovitz and Shapira \cite{Moshkovitz2012RamseyTI} proved the following result.
\begin{theorem}\cite{Moshkovitz2012RamseyTI} For every $t,n \ge 2$,
\[2^{\frac{2}{3}t^{n-1}/\sqrt n} \le M_3(n,t) \le 2^{2t^{n-1}}.\]
\end{theorem}
\nin The proof of the above theorem uses the fascinating relationship
\[\alpha([t]^n)=M_3(n,t)-1 \quad \text{for all }  t,n \ge 2\] 
that is also proved in \cite{Moshkovitz2012RamseyTI}.  
Moshkovitz and Shapira conjectured that
\[M_3(n,t)=2^{\Theta(t^{n-1}/\sqrt n)},\]
and more boldly, asked whether (as $n \rightarrow \infty$)
\beq{MS.bold} M_3(n,t)=2^{(1+o_n(1))N(t,n)}.\enq
Of course, the motivation of this bold question is the fascinating phenomenon proved by Kleitman and Markowsky that almost all antichains in the Boolean lattice are subsets of the middle layer(s). Our second main result, 
Theorem~\ref{thm:t^n}, confirms that \eqref{MS.bold} holds for $t,n$ with $t \ll \frac{n}{\log^3 n}$ as $n \rightarrow \infty$ ($t$ is not necessarily fixed).

\medskip

\nin \textit{Organization.} Section \ref{sec.entropy} collects basic properties of entropy. The two main theorems are proved in Sections \ref{sec.3^n} and \ref{sec.t^n}, respectively.

	\section{Entropy Basics}\label{sec.entropy}

    For $p\in [0,1]$, we denote by $H(p)=-p\log p-(1-p)\log(1-p)$ (where $0\log0:=0$) the \textit{binary entropy function}.    For a discrete {random variable} $X$, we define its \textit{entropy} as
    $$H(\bX)=\sum_x -p(x)\log p(x),$$
    where $p(x)=\pr(\bX=x)$. Note that by Jensen's inequality we get
    \beq{max.unif} H(\bX)\leq \log |\text{range}(\bX)| \quad \text{(equality holds iff $\bX$ is uniform on range($\bX$)).}\enq    The \textit{conditional entropy} of $\bX$ with respect to a discrete random variable $\bY$ is
    $$H(\bX|\bY=y)=\sum_x -p(x|y)\log p(x|y)$$
    and
    \beq{entropy.basic1} H(\bX|\bY)=\E(H(\bX|\bY=y))= \sum_y p(y)\sum_x -p(x|y)\log p(x|y),\enq
    where $p(x|y)=\pr(\bX=x|\bY=y)$. For the conditional entropy, we have
    $$H(\bX,\bY)=H(\bX)+H(\bY|\bX),$$
    and if $\bY$ determines $\bX$, as in $H(\bX|\bY) = 0$, then
    $$H(\bX)\leq H(\bY).$$
    Using the above, it is straightforward to prove the following:
    
    \begin{fact}[Subadditivity of entropy]
    For any random vector $\bX=(\bX_1,\dots,\bX_n)$ we have
    $$H(\bX)\leq H(\bX_1)+\dots+H(\bX_n).$$
    \end{fact}

    We will also need a celebrated inequality due to Shearer, which generalizes the subadditivity property of entropy.

	\begin{lemma}[\cite{Graham1986SomeIT}]\label{Sh}
		If $\bX=(\bX_1, \ldots, \bX_k)$ is a random vector and $\alpha:2^{[k]} \rightarrow \mathbb R^+$ satisfies
		$\sum_{A \ni i} \alpha_A\ge 1$ for all $i \in [k]$,
		then
		\[H(\bX)\le \sum_{A \subseteq [k]} \alpha_A H(\bX_A),\]
		where $\bX_A=(\bX_i:i \in A)$.
	\end{lemma}
    \nin A derivation of the below elementary fact can be found in \cite{Kahn2002Entropy}.
    \begin{fact}\label{fact}
		\[H(\bX|\bX \ne 0)=\frac{H(\bX)-H(\pr(\bX=0))}{1-\pr(\bX=0)}.\]	
	\end{fact}

\begin{comment}
    	\begin{proof}
    		For $A=\{i_1, \ldots, i_\ell\}$ with $i_1<i_2<\ldots <i_\ell$, we have
    		\[\begin{split}
    			H(X_A) &=H(X_{i_1}, \ldots, X_{i_\ell})\\
    			&= H(X_{i_1})+H(X_{i_2}|X_{i_1})+\ldots+H(X_{i_\ell}|X_{i_1}, \ldots, X_{i_{\ell-1}})\\
    			& \ge H(X_{i_1}|X_1,\ldots, X_{i_1-1})+H(X_{i_2}|X_1,\ldots,X_{i_2-1})+\ldots+H(X_{i_\ell}|X_1,\ldots,X_{i_\ell-1}).\\
    		\end{split}\]
    		So we have
    		\[\begin{split}\sum_{A \subseteq [k]} \alpha_A H(X_A) & \ge \sum_{A \sub [k]} \alpha_A \sum_{i \in A} H(X_i|X_1,\ldots,X_{i-1})\\
    			& = \sum_{i \in [k]}\sum_{A \ni i} \alpha_A H(X_i|X_1,\ldots,X_{i-1})\\
    			&\ge \sum_{i \in [k]} H(X_i|X_1,\ldots,X_{i-1}) = H(X).
    		\end{split}\]
    	\end{proof}
\end{comment}

				\begin{proposition}\label{prop2.4}
				For any set $B$, $\lambda:B \rightarrow [0,\infty)$, $\mathcal S \sub 2^B$, and probability distribution $p$ on $\mathcal S$,
				\beq{13}\sum_{S \in \mathcal S} p_S \Bigl(\log(1/p_S)+\sum_{x \in S}\log \lambda_x\Bigr)\le \log  \left( \sum_{S \in \mathcal S} \prod_{x \in S} \lambda_x\right).\enq
			\end{proposition}
			
			\begin{proof}
				Set $W=\sum_{S \in \mathcal S} \prod_{x \in S} \lambda_x$ and $q_S=\prod_{x \in S} \lambda_x/W$. Then the left-hand side of \eqref{13} is $\sum_{S \in \mathcal S} p_S \log (q_S/p_S)+\log W$. By Jensen's inequality, $\sum_{S \in \mathcal S} p_S \log (q_S/p_S)\leq \log\left(\sum_{S \in \mathcal S}q_S\right)=0$, from which the conclusion follows. \end{proof}

	\section{Proof of Theorem \ref{thm:3^n}}\label{sec.3^n}

	We start with an inequality for the independence polynomial of a bipartite graph, where we allow the vertices of one of the parts to have different activities. 
The format and the proof of the next theorem are inspired by \cite{Kahn2002Entropy}, but the irregularity of the degrees of $v \in [3]^n$ required more careful analyses and additional ideas.
	
	\begin{theorem}\label{thm:two_level} Let $G$ be a bipartite graph on $A \cup B$ with the weight for each vertex $x \in A \cup B$ defined to be $\lambda_x \ge 1$.
		For $v \in A$, let
		\[\dn{v}=\min\{d(u):u \in N(v)\}\]
		and assume that the weight on each $v \in A$ is
		\[\lambda_v \equiv \mu\,,\]
  for some $\mu$. Then
			\[\sum_{I \in \cI(G)} \prod_{x \in I} \lambda_x \le \prod_{v \in A} \Bigl[ (1+\mu)^{d_{N(v)}}+\prod_{u \in N(v)}(1+\lambda_u)-1\Bigr]^{1/d_{N(v)}}.\]
		\end{theorem}
		
				We use the setup in \cite{Kahn2002Entropy} for our proof. For the time being, $G=(V, E)$ is an arbitrary graph and $\lambda:V \rightarrow [1,\infty)$ is an arbitrary assignment of weights to the vertices. Write $\mathcal I(G)$ for the collection of independent sets in $G$. Set
		\[Z=Z(G,\lambda)=\sum_{I \in \cI(G)} \prod_{x \in I} \lambda_x.\]

		\nin With each $v \in V$ we associate a set $S_v \ni 0$ and nonnegative weights $\alpha_v(s)$, $s \in S_v$, such that
		\[\alpha_v(0)=1, \sum_{s \ne 0} \alpha_v(s)=\lambda_v,\]
		and the r.v. $\bX_v$ given by $\pr(\bX_v=s)=\frac{\alpha_v(s)}{1+\lambda_v}$
		satisfies 
  \beq{eq.HXv} H(\bX_v)=\log(1+\lambda_v).\enq (As noted in \cite{Kahn2002Entropy}, this is possible iff $\lambda_v \ge 1$.) 
		We say that a vector $(s_v:v \in V) \in \prod S_v$ is \textit{independent} if $\{v:s_v \ne 0\} \in \cI(G).$  
	    Finally, let $\bY=(\bY_v:v \in V)$ be chosen from the independent vectors in $\prod S_v$ with $\pr(\bY=(s_v)) \propto \prod \alpha_v(s_v)$. 
     It was proved  in \cite{Kahn2002Entropy} that
		\[ H(\bY)=\log Z.\]
		
		\begin{proof}[Proof of Theorem \ref{thm:two_level}]
			
			Let $G$ and the values $\lambda_x$ be as in the statement of the theorem, and $\bY$ as above. We must show that
			\beq{9}H(\bY) \le \sum_{v \in A} \frac{1}{\dn{v}} \log \Bigl[(1+\mu)^{\dn{v}}+\prod_{u \in N(v)} (1+\lambda_u)-1\Bigr].\enq
			Denote by $Q_v$ the event $\{\bY_w=0 \; \forall w \in N(v)\}$, and set $q_v=\pr(Q_v)$. By the definition of $\bY$, we have
   \beq{XY} H(\bY_v|Q_v)=H(\bX_v).\enq
			
			Writing $\bY_W$ for $(\bY_w:w \in W)$, we have, using Lemma \ref{Sh},
			\beq{Thm3.1.1} \begin{split}H(\bY) &=H(\bY_A|\bY_B)+H(\bY_B)\\
				& \le \sum_{v\in A} \left[H(\bY_v|\bY_B)+\frac{1}{d_{N(v)}}H(\bY_{N(v)})\right].
			\end{split}\enq

			The entropy terms in the above sum are
			\beq{Thm3.1.2} H(\bY_v|\bY_B)\stackrel{\eqref{entropy.basic1}}{=} q_v H(\bY_v|Q_v)\stackrel{\eqref{eq.HXv}, \eqref{XY}}{=}q_v\log(1+\mu)\enq
			and (with $\mathbf{1}_A$ the indicator of $A$)
			\beq{Thm3.1.3}H(\bY_{N(v)})=H(\bY_{N(v)}, \mathbf{1}_{Q_v})=H(\mathbf{1}_{Q_v})+H(\bY_{N(v)}|\mathbf{1}_{Q_v})=H(q_v)+(1-q_v)H(\bY_{N(v)}|\ov{Q_v}).\enq

   \begin{claim}
       For any $v \in A$,
       			\beq{10} H(\bY_{N(v)}|\ov{Q_v})\le \log\left(\prod_{u \in N(v)}(1+\lambda_u)-1\right).\enq
   \end{claim}

 \begin{proof}
Let $\bI$ be a random independent set from $\cI(G)$ chosen with $P(\bI=I)=\prod_{v \in I} \lambda_v/Z,$ and  $\bT=\bI \cap N(v)$. Observe that $H(\bY_{N(v)}|\bT=T)=\sum_{u \in T} H(\bX_u|\bX_u\ne 0)=\sum_{u \in T} \log \lambda_u$. Thus, setting $p_T=\pr(\bT=T|\ov{Q_v})$,
			\[H(\bY_{N(v)}|\ov{Q_v})=H(\bY_{N(v)}, \bT|\ov{Q_v})=H(\bT|\ov{Q_v})+H(\bY_{N(v)}|\bT,\ov{Q_v})=\sum_{T \ne \emptyset} p_T\left(\log(1/p_T)+\sum_{u \in T}\log \lambda_u \right).\]
By Proposition \ref{prop2.4}, this is at most
\[\log\left(\sum_{\emptyset \ne T \sub N(v)}\prod_{u \in T} \lambda_u \right) = \log\left(\prod_{u \in N(v)} (1+\lambda_u)-1\right).\] \end{proof}

			The combination of \eqref{Thm3.1.1}-\eqref{10} gives
			\beq{11}H(\bY) \le \sum_{v \in A} \left[q_v\log(1+\mu)+\frac{1}{d_{N(v)}} \left\{H(q_v)+(1-q_v)\log\left(\prod_{u \in N(v)}(1+\lambda_u)-1\right)\right\}\right].\enq
			The contribution of $v$,
			\[\frac{1}{d_{N(v)}}\log\left(\prod_{u \in N(v)}(1+\lambda_u)-1\right)+\frac{1}{d_{N(v)}}\left[H(q_v)+q_v\left\{d_{N(v)}\log(1+\mu)-\log\left(\prod_{u \in N(v)}(1+\lambda_u)-1\right)\right\}\right],\]
			is maximized at
			\[q_v=\frac{2^T}{2^T+1}=\frac{(1+\mu)^{d_{N(v)}}}{(1+\mu)^{d_{N(v)}}+\prod_{u \in N(v)}(1+\lambda_u)-1},\]
			where $T=d_{N(v)}\log(1+\mu)-\log[\prod_{u \in N(v)}(1+\lambda_u)-1]$.
			
			Inserting this value of $q_v$ in \eqref{11} gives
			\[\log Z= H(\bY)\le \sum_{v \in A} \frac{1}{d_{N(v)}}\log\Bigl[(1+\mu)^{d_{N(v)}}+\prod_{u \in N(v)} (1+\lambda_u)-1\Bigr].\]

		\end{proof}
			
	Let $P$ be a graded poset with levels $P_1,P_2,\dots, P_k$. For $X\subseteq P_j$ for some $j$, write $P_{<X}=\{z\in P: z<v\text{ for some } v\in X\}$, $M(X)=P_j \setminus X$, $P_X=\{z\in P_1\cup\dots\cup P_{j-1}: z\not< x\ \forall x\in X\}$, and $\bar P_X=P_X \cup M(X)$. If $X$ is a singleton $\{v\}$, then we simply write $P_{<v}$ for $P_{<X}$. Finally, let $N^i(v)=P_i \cap \{y:y>v \mbox{ or } y<v\}$, $d^i(v)=|N^i(v)|$ and $d_{N^i(v)}=\min\{d^{i+1}(w):w \in N^i(v)\}$. When not working with the main poset $P$, we will be adding subscripts to the notation to keep track of the poset we are considering.
 
	For a poset $P$ with $k$ levels and $\lambda_1,\ldots,\lambda_k \ge 0$, define $f_P(\lam_1,\dots,\lam_k)$ recursively as follows:
    for any poset $Q$ with one level, $f_Q(\lam_1):=(1+\lam_1)^{|Q|}$. If $P$ has $k\geq 2$ layers, then
	$$f_P(\lam_1,\dots,\lam_k):=\prod_{v\in P_k}\left((1+\lam_k)^{d_{N^{k-1}(v)}}+f_{P_{<v}}(\lam_1,\dots,\lam_{k-1})-1\right)^{1/d_{N^{k-1}(v)}}.$$
 For example, 
 \beq{2 level f} \mbox{if $P=P_1\cup P_2$, then 
$f_P(\lam_1,\lam_2)=\prod_{v\in P_2}\left((1+\lam_2)^{d_{N^1(v)}}+(1+\lam_1)^{d^1(v)}-1\right)^{1/d_{N^1(v)}}$.}\enq
    Observe that the assumption $\lam_i\geq0$ easily yields 
    \beq{f.ge.1} \mbox{$f_P \ge 1$\,, \ for any poset $P$.}\enq

    \begin{proposition}\label{lem:incr}
        Let $P$ be a graded poset with $k$ levels, and $\lambda_1,\ldots, \lambda_k \ge 0$. For any $Y\subseteq P_k$, 
        $f_{\overline{P}_Y}(\lam_1,\dots,\lam_k)\leq f_P(\lam_1,\dots,\lam_k).$
    \end{proposition}
    \begin{proof}
        Let $Q=\overline{P}_Y$. We use induction on the number of levels $k$. The assertion trivially holds for $k=1$. For $k\geq 2$,
        \begin{align*}
            f_{Q}(\lam_1,\dots,\lam_k)&=\prod_{v\in Q_{k}}\left((1+\lam_{k})^{d_{N_{Q}^{k-1}(v)}}+f_{Q_{<v}}(\lam_1,\dots,\lam_{k-1})-1\right)^{1/d_{N_{Q}^{k-1}(v)}}\\
            &\leq \prod_{v\in Q_{k}}\left((1+\lam_{k})^{d_{N^{k-1}(v)}}+f_{Q_{<v}}(\lam_1,\dots,\lam_{k-1})-1\right)^{1/d_{N^{k-1}(v)}},
        \end{align*}
        where the inequality uses the facts that $\frac{1}{x}\log(A^x+B)$ ($A,B \ge 0$) is decreasing for $x>0$ and that $d_{N_Q^{k-1}}(v)\geq d_{N^{k-1}(v)}$ for any $v \in Q_k$. Furthermore, the induction hypothesis yields $f_{Q_{<v}}(\lambda_1,\ldots,\lambda_{k-1})\leq f_{P_{<v}}(\lambda_1,\ldots,\lambda_{k-1})$ for any $v \in Q_k$ (because, with $Z:=N^{k-1}(v)\cap P_{<Y}$, we have $Q_{<v}=\overline{\left(P_{<v}\right)}_{Z}$). Therefore,
        \begin{align*}
        f_{Q}(\lam_1,\dots,\lam_k)&\leq \prod_{v\in Q_{k}}\left((1+\lam_{k})^{d_{N^{k-1}(v)}}+f_{P_{<v}}(\lam_1,\dots,\lam_{k-1})-1\right)^{1/d_{N^{k-1}(v)}}\\
        &\leq \prod_{v\in P_{k}}\left((1+\lam_{k})^{d_{N^{k-1}(v)}}+f_{P_{<v}}(\lam_1,\dots,\lam_{k-1})-1\right)^{1/d_{N^{k-1}(v)}}=f_P(\lam_1,\dots,\lam_k)
        \end{align*}
        (the last inequality holds because each extra term is $\geq1$).
    \end{proof}

    We are now ready to prove the main theorem of this section.
	\begin{theorem}\label{thm:multi_level}
		Let $P$ be a graded poset with levels $P_1,P_2,\dots, P_k$, and $\cA(P)$ be the collection of antichains of $P$. For each $x \in P_i$, define $\lam_x \equiv \lam_i$ where $\lam_j\geq 1$ for all $1\leq j\leq k$. Then,
		\beq{thm.multi} \sum_{I\in \mathcal{A}(P)}\prod_{x\in I}\lam_x\leq f_P(\lam_1,\dots,\lam_k).\enq
	\end{theorem}
	\begin{proof}
		
		We proceed by induction on $k$. The base case ($k=2$) follows from Theorem \ref{thm:two_level} (see \eqref{2 level f}). Assume the theorem is true for any poset with $k-1$ levels and let $P$ be a poset with $k$ levels. Then for any $X\subseteq P_k$, we have
  \[\sum_{\substack{I\in \mathcal{A}(P)\\ I\cap P_k=X}} \prod_{x\in I}\lam_x=\lam_k^{|X|}\sum_{I\in \mathcal{A}(P_X)}\prod_{x\in I}\lam_x
			\leq \lam_k^{|X|}f_{P_X}(\lam_1,\dots,\lam_{k-1}).\]
		Therefore,
		\begin{align*} 
			\sum_{I\in \mathcal{A}(P)}\prod_{x\in I}\lam_x&\leq\sum_{X\subseteq P_k}\lam_k^{|X|}f_{P_X}(\lam_1,\dots,\lam_{k-1})\\
			&=\sum_{X\subseteq P_k}\lam_k^{|X|}\prod_{v\in (P_X)_{k-1}}\left((1+\lam_{k-1})^{d_{N_{P_X}^{k-2}(v)}}+f_{(P_X)_{<v}}(\lam_1,\dots,\lam_{k-2})-1\right)^{1/d_{N_{P_X}^{k-2}(v)}}\\
			&\stackrel{(*)}{\leq} \sum_{X\subseteq P_k}\lam_k^{|X|}\prod_{v\in (P_X)_{k-1}}\left((1+\lam_{k-1})^{d_{N^{k-2}(v)}}+f_{(P_X)_{<v}}(\lam_1,\dots,\lam_{k-2})-1\right)^{1/d_{N^{k-2}(v)}}\\
            &\leq \sum_{X\subseteq P_k}\lam_k^{|X|}\prod_{v\in (P_X)_{k-1}}\left((1+\lam_{k-1})^{d_{N^{k-2}(v)}}+f_{P_{<v}}(\lam_1,\dots,\lam_{k-2})-1\right)^{1/d_{N^{k-2}(v)}}
		\end{align*}
        where ($*$) uses the facts that $\frac{1}{x}\log(A^x+B)$ ($A,B \ge 0$) is decreasing for $x>0$ and that $d_{N_{P_X}^{k-2}}(v)\geq d_{N^{k-2}}(v)$ for any $v \in (P_X)_{k-1}$; the last inequality follows from Proposition \ref{lem:incr} (applied to $P_{<v}$ with $Y=(P_{<v})_{k-2}\cap P_{<X}$). For each $v \in P_{k-1}$, define
        \beq{def.mu} \mu_{v}=\left((1+\lam_{k-1})^{d_{N^{k-2}(v)}}+f_{P_{<v}}(\lam_1,\dots,\lam_{k-2})-1\right)^{1/d_{N^{k-2}(v)}}-1.\enq
        We have shown that
        $$\sum_{I\in \mathcal{A}(P)}\prod_{x\in I}\lam_x\le\sum_{X\subseteq P_k}\lam_k^{|X|}\prod_{v\in (P_X)_{k-1}}(1+\mu_{v}).$$        
        Recall from \eqref{f.ge.1} that $f_{Q}\geq1$ for any graded poset $Q$, so
        $$\mu_{v}\geq \left((1+\lam_{k-1})^{d_{N^{k-2}(v)}}\right)^{1/d_{N^{k-2}(v)}}-1=\lam_{k-1}\geq 1.$$
		Now, by applying Theorem \ref{thm:two_level} with $A=P_k,\ B=P_{k-1}$, $\mu=\lambda_k$, and $\mu_v$ in \eqref{def.mu} the weight for each $v \in B$, we get
		\begin{align*}
			&\sum_{X\subseteq P_k}\lam_k^{|X|}\prod_{v\in (P_{X})_{{k-1}}}(1+\mu_v)\leq \prod_{v\in P_k}\left((1+\lam_k)^{d_{N^{k-1}(v)}}+\prod_{u \in N^{k-1}(v)}(1+\mu_u)-1\right)^{1/d_{N^{k-1}(v)}}\\
			&\stackrel{\eqref{def.mu}}{=}\prod_{v\in P_k}\left((1+\lam_k)^{d_{N^{k-1}(v)}}+\prod_{u\in N^{k-1}(v)}\left((1+\lam_{k-1})^{d_{N^{k-2}(u)}}+f_{P_{<u}}(\lam_1,\dots,\lam_{k-2})-1\right)^{1/d_{N^{k-2}(u)}}-1\right)^{1/d_{N^{k-1}(v)}}\\
			&=\prod_{v\in P_k}\left((1+\lam_k)^{d_{N^{k-1}(v)}}+f_{P_{<v}}(\lam_1,\dots,\lam_{k-1})-1\right)^{1/d_{N^{k-1}(v)}}\\
			&=f_P(\lam_1,\dots,\lam_k),
		\end{align*}
        which completes the induction.
	\end{proof}
    
     For the rest of this section, we work specifically on $[3]^n$. Write $P_0, P_1,\dots, P_{2n}$ for the levels of $[3]^n$. Note that $[3]^n$ enjoys the special property that
     \beq{3n.special} d_{N^{i-1}(x)}-d^{i-1}(x) \ge n-i \quad \forall x \in P_i.\enq
     Let $P=P_0\cup P_1\cup\dots\cup P_n$ be the bottom half. Recall that $\alpha(P)$ denotes the number of antichains in $P$.
    
    \begin{lemma}\label{prop:bottom}
        For any $Y\subseteq P_n$, 
        \beq{py.anti} \alpha\left(\overline{P}_Y\right)\leq 2^{|M(Y)|\left(1+\frac{2\log 3}{n}\right)}.\enq
    \end{lemma}
	\begin{proof}
        Let $R:=\overline{P}_Y$ and $R_0,R_1,\dots, R_n$ be its levels. We will prove that
		\begin{equation}\label{eq:half_ind}
			f_{R_{<v}}(1,\dots,1)\leq 2^{d^{j-1}(v)\left(1+1/d_{N^{j-1}(v)}\right)} \quad \forall v \in R_j.
		\end{equation}
		We first show that \eqref{eq:half_ind} implies the lemma. We have
		\begin{align*}
			\alpha(R) \stackrel{\eqref{thm.multi}}{\le} f_R(1,\dots,1)&=\prod_{v\in R_n}\left(2^{d_{N_R^{n-1}(v)}}+f_{R_{<v}}(1,\dots,1)-1\right)^{1/d_{N_R^{n-1}(v)}}\\
   &\leq\prod_{v\in R_n}\left(2^{d_{N^{n-1}(v)}}+f_{R_{<v}}(1,\dots,1)\right)^{1/d_{N^{n-1}(v)}}\\
			&\stackrel{\eqref{eq:half_ind}}{\leq} \prod_{v\in R_n}\left(2^{d_{N^{n-1}(v)}}+2^{d^{n-1}(v)\left(1+1/d_{N^{n-1}(v)}\right)}\right)^{1/d_{N^{n-1}(v)}}\\
   &\stackrel{\eqref{3n.special}}{\le} \prod_{v \in R_n} 2\cdot 3^{1/d_{N^{n-1}(v)}}=\prod_{v \in R_n} 2^{1+\log 3/d_{N^{n-1}(v)}}.
		\end{align*}
Finally, the lemma follows by noticing that $d_{N^{n-1}(v)}\geq n/2$ for any $v\in R_n$.
		
		\nin \textit{Proof of (\ref{eq:half_ind}).} For $j=1$, for any $v \in R_1$, we trivially have $f_{R<v}(1) \le 2^{d^{0}(v)}<2^{d^{0}(v)\left(1+1/d_{N^{0}(v)}\right)}$. Next, assume the statement is true for any vertex up to layer $j-1$. For any $v\in R_j,$ 
		\begin{align*}
			f_{R_{<v}}(1,\dots,1)&=\prod_{u\in N^{j-1}(v)}\left(2^{d_{N^{j-2}(u)}}+f_{R_{<u}}(1,\dots,1)-1\right)^{1/d_{N^{j-2}(u)}}\\
			&\stackrel{(\dagger)}{\leq} \prod_{u\in N^{j-1}(v)}\left(2^{d_{N^{j-2}(u)}}+2^{d^{j-2}(u)\left(1+1/d_{N^{j-2}(u)}\right)}\right)^{1/d_{N^{j-2}(u)}}\\
   &\leq \prod_{u\in N^{j-1}(v)}2\left(1+2^{d^{j-2}(u)-d_{N^{j-2}(u)}+d^{j-2}(u)/d_{N^{j-2}(u)}}\right)^{1/d_{N^{j-2}(u)}}\\
   & \stackrel{\eqref{3n.special}}{\le} \prod_{u\in N^{j-1}(v)}2^{1+1/d_{N^{j-2}(u)}}
		\end{align*}
  (where $(\dagger)$ uses the induction hypothesis). But for every $u\in N^{j-1}(v)$, we have $d_{N^{j-2}(u)}\geq d_{N^{j-1}(v)}$, hence the above expression is at most $2^{d^{j-1}(v)\left(1+1/d_{N^{j-1}(v)}\right)}$.		
	\end{proof}

    \begin{proof} [Proof of Theorem \ref{thm:3^n}]
        For $I\in \cA(P_{n+1} \cup \ldots \cup P_{2n})$, let $X=X(I) \sub P_n$ be the "lower shadow" of $I$ on $P_{n}$, namely, $X=\{v \in P_{n}:v< w \in I\}$. Then        
        \[\begin{split} \alpha([3]^n) &= \sum_{Y \sub P_n}|\{I\in \cA(P_{n+1} \cup \ldots \cup P_{2n}):X(I)=Y\}|\cdot|\cA(\bar P_Y)|\\
        & \stackrel{\eqref{py.anti}}{\le} \sum_{Y \sub P_n}2^{(1+2\log 3/n)|M(Y)|}|\{I\in \cA(P_{n+1} \cup \ldots \cup P_{2n}):X(I)=Y\}|\\
        & \le 2^{2\log 3|P_n|/n}\sum_{Y \sub P_n}2^{|M(Y)|}|\{I\in \cA(P_{n+1} \cup \ldots \cup P_{2n}):X(I)=Y\}|\\
        & = 2^{2\log 3|P_n|/n} |\cA(P_n \cup \ldots \cup P_{2n})|\\
        & \stackrel{\eqref{py.anti}}{\le} 2^{2\log 3|P_n|/n}2^{(1+2\log 3/n)|P_n|}=2^{(1+4\log 3/n)|P_n|}.
        \end{split}\]
    \end{proof}

	\section{Proof of Theorem \ref{thm:t^n}}\label{sec.t^n}

The proposition below easily follows by adapting the proof of \cite[Theorem 1]{Kleitman1976OnJCP}. We put its short proof in the appendix  for ease of reference. We use $y \lessdot x$ to denote that $y$ is immediately below $x$ in $[t]^n$.

 \begin{proposition}\label{GK_extend}
     For any $t, n \ge 1$, the poset $[t]^n$ admits a chain partition $\mathcal{C}=\{C_1,C_2,\dots,C_N\}$ of size $N=N(t,n)$ that satisfies the following property: 
         \beq{WMA} \text{if $y$ is immediately below $x$ in a chain $C \in \cC$, then $y \lessdot x$ in $[t]^n$.}\enq
     
 \end{proposition}

    We also recall from \cite{Moshkovitz2012RamseyTI} the following easy lower bound on $N(t,n)$ that works for all $t$ and $n$.

    \begin{lemma}[Lemma 2.6, \cite{Moshkovitz2012RamseyTI}]\label{lem.MS}
        For all $t, n \ge 1$,
        \[N(t,n) \ge \frac{2t^{n-1}}{3\sqrt n}.\]
    \end{lemma}

    An element $x \in [t]^n=\{0,1,\ldots,t-1\}^n$ is called a \textit{point}. For $l \in \{0,1,\ldots,t-1\}$, let $d_l(x)$ be the number of coordinates of $x$ equal to $l$. We say $x$ is \textit{low} if $d_l(x)<\frac{n}{2t}$ for some $1\leq l\leq t-1$ and \textit{high} otherwise. We also say a chain $C_j\in\mathcal{C}$ is \textit{low} if it contains a low point and \textit{high} otherwise.

	Let $(\cF(t,n),\prec)$ be the poset on the family of monotone Boolean functions on $[t]^n$ where $f \prec g$ iff $f(x) \le g(x) \ \forall x \in [t]^n$. Observe that $|\cF(t,n)|=\alpha([t]^n)$. Thus, with $\bff$ a uniformly chosen element of $\cF(t,n)$, we have
 \[\log\alpha([t]^n)\stackrel{\eqref{max.unif}}{=}H(\bff).\]
 Following the approaches in \cite{Pippenger1999EntropyAE} and \cite{Carroll2012CountingAA}, we will define a series of random variables that determine $\bff$ in order to bound $H(\bff)$. In what follows, we use bold-face letters ($\mathbf f, \mathbf{y}, \mathbf {Y}, \ldots)$ for random variables, while plain letters ($f, y, Y, \ldots)$ represent values that the corresponding random variable takes.

 For $f \in \cF=\cF(t,n)$ and $j \in [N]$, let $\gamma_j(f)=|\{x\in C_j:f(x)=1\}|.$ Note that $(\gamma_j(f))_j$ determines $f$. But exposing $\gamma_j(f)$ for all $j$ is too expensive, so we will make a random choice on which chains to expose. To that end, first define  $\mathbf y_j$ as follows. Let $p=\left[t\log((t-1)n)/n\right]^{1/2}$.  If the chain $C_j$ is high, then
$\pr(\mathbf y_j=1)=p=1-\pr(\mathbf y_j=0)$; if $C_j$ is low, then $\mathbf y_j \equiv 1$. Having $\mathbf y_j=1$ means we will expose $\gamma_j(f)$. Define $\tilde \bY_j(f)=\mathbf y_j\gamma_j(f)$. 

As in \cite{Pippenger1999EntropyAE,Carroll2012CountingAA}, in order to complement $(\tilde \bY_j(f))_j$, we introduce another random variable. Write $\tilde \bY=(\tilde \bY_j)_j$ (and similarly for $\hat \bY$ later). Given $\tilde \bY(f)=\tilde Y$, let $\tilde{f}$ be the smallest in $\cF$ that satisfies $\gamma_j(\tilde{f})\geq\tilde{Y}_j$ for all $1\leq j\leq N$, and set $\hat{Y}_j=\gamma_j(f)-\gamma_j(\tilde{f}) \ (\ge 0).$  Observe that $f$ is determined by the pair $(\tilde Y(f), \hat Y(f))$, so in particular,
       $$H(\bff)\leq H(\tilde \bY(\bff), \hat \bY(\bff)) \le H(\tilde \bY(\bff))+H(\hat \bY(\bff)).$$
(Note that, given $f$, the randomness of $\tilde \bY(f)$ and $\hat \bY(f)$ is inherited from $\mathbf y$.)
        Therefore, the next two assertions complete the proof of Theorem \ref{thm:t^n}. Here and for the rest of the section, we will use $\eps, \eps', \ldots$ for small (absolute) constants. 
 \beq{lem.tilde Y} H(\tilde \bY(\bff)) \le (2+\eps)N \frac{t^{1/2}(\log((t-1)n))^{3/2}}{n^{1/2}},\enq
 and
 \beq{lem.hat Y} H(\hat \bY(\bff)) \le   N\left(1+\frac{(4+2\eps')t^{1/2}(\log(t^{1/2}(t-1)n))^{3/2}}{n^{1/2}}\right).\enq

The following handy lemma is given in \cite{Pippenger1999EntropyAE}. 
    \begin{lemma}[Pippenger \cite{Pippenger1999EntropyAE}]\label{lem:pip}
		Suppose the random variable $\mathbf K$ takes values in $\{0,1,\dots,n\}$, and $P(\mathbf K \geq k) \leq q$ for some $k\geq1$ and $0\leq q\leq1$.
		Then 
  \[H(\mathbf K) \leq h_1(q) + \log k + q \log n,\]
  where
  \[h_1(q)=\begin{cases}
    -q\log q-(1-q)\log(1-q) & \text{if} \quad q \in [0,1/2];\\
    1 & \text{if} \quad q \in [1/2,1].
\end{cases}\]
	\end{lemma}
For future reference, we remark that
 \beq{h1.bound} \text{$h_1(q) \le -2q\log q$ \; for $q \le 1/2$}.\enq

  \begin{proof}[Proof of \eqref{lem.tilde Y}]  If $C_j$ is low, then we apply the naive bound $H(\tilde \bY_j(\bff)) \le \log((t-1)n+2).$ If $C_j$ is high, then noticing that $\tilde \bY_j(\bff) \ge 1$ implies $\mathbf y_j=1$, we apply Lemma \ref{lem:pip} with $k=1$ and $q=p$ to obtain
 \[H(\tilde \bY_j(\bff)) \le h_1(p)+p\log ((t-1)n+1).\]
 Therefore, with $M$ the number of low chains in $\mathcal C$,
 \beq{tilde.Y.bound} H(\tilde{\bY}(\bff))\leq\sum_{j=1}^{N}H(\tilde{\bY}_j(\bff)) \le M\log((t-1)n+2)+(N-M)(h_1(p)+p\log((t-1)n+1)).\enq
 
 Since each low chain contains a low point, we may bound $M$ by the number of low points. Any such point satisfies $d_l(x)<\frac{n}{2t}$ for some $l \in \{1,\ldots,t-1\}$. Let $\mathbf x$ be a uniformly random point in $[t]^n$; equivalently, each coordinate of $\mathbf x$ is chosen uniformly and independently from $\{0,\ldots,t-1\}$. Then for each $l \in [t-1],$ the Chernoff bound yields that $\pr(d_l(\mathbf x)<n/(2t))\leq e^{-n/(8t)}$, 		hence by the union bound,
		\beq{bound.M} M\leq (t-1) t^ne^{-n/(8t)}\le t^{n+1}e^{-n/(8t)}.\enq
By combining \eqref{h1.bound} (note $p \le 1/2$ since $t \le n/(100\log n)$), \eqref{tilde.Y.bound} and \eqref{bound.M}, we have
		\begin{align*}
			H(\tilde{\bY}(\bff))&\leq  t^{n+1}e^{-n/(8t)}\log((t-1)n+2)+N(-2p\log p+p\log((t-1)n+1))\\
   & \le (2+\eps)N \frac{t^{1/2}(\log((t-1)n))^{3/2}}{n^{1/2}}\,,
		\end{align*} 
  where the last inequality uses Lemma \ref{lem.MS} (and that $t \le n/(100\log n)$).
  \end{proof}

	\begin{proof}[Proof of \eqref{lem.hat Y}] For each $j \in [N]$, let $q_j=\pr(\hat \bY_j (\bff) \ge 2)$. By Lemma \ref{lem:pip} with $k=2$ and $q=q_j$, $H(\hat \bY_j(\bff)) \le h_1(q_j)+1+q_j\log((t-1)n).$ Therefore,
with  $Q:=\sum_{j=1}^Nq_j$,
\[\begin{split} H(\hat \bY(\bff))\le \sum_j H(\hat \bY_j(\bff)) &\le \sum_j 
 [h_1(q_j)+1+q_j\log((t-1)n)]\\
& \stackrel{(*)}{\le} Nh_1(Q/N)+N+Q\log((t-1)n)\,,\end{split}\]
where $(*)$ follows from the concavity of $h_1$ (so $\frac{1}{N}\sum_j h_1(q_j) \le h_1\left(\sum_j q_j/N\right)$). We will prove that
\beq{Q.bound} Q\le \frac{(2+\eps')t^{1/2}(\log (t^{1/2}(t-1)n))^{1/2}}{n^{1/2}}N.\enq
Let's first show that \eqref{Q.bound} implies \eqref{lem.hat Y}. By \eqref{h1.bound} (and $Q/N \le 1/2$),
		\begin{align*}
			h_1\left(\frac{Q}{N}\right)&\leq \frac{(2+\eps')t^{1/2}(\log(t^{1/2}(t-1)n))^{1/2}}{n^{1/2}}\log n,
   \end{align*}   
   thus
		\begin{align*}
			H(\hat{\bY}(\bff))&\leq  
   N\left(1+\frac{(4+2\eps')t^{1/2}(\log(t^{1/2}(t-1)n))^{3/2}}{n^{1/2}}\right).
		\end{align*}

  The rest of this section is devoted to proving \eqref{Q.bound}. To this end, we first define an event that is implied by $\hat Y_j(f) \ge2$. Given $f$ and $\tilde Y(f)$, say a point $x \in C_j$ is \textit{bad} if 
   \begin{enumerate}[(i)]
       \item  $x$ is high (that is, $d_l(x) \ge \frac{n}{2t}$ for all $1 \le l \le t-1$); and
    \item $\tilde{f}(x)=0$; and
       \item $f(y)=1$ for $y \in C_j$ that is immediately below $x$.
         \end{enumerate} 
Define $r_x$ to be the probability (w.r.t. $\bff$ and $\mathbf y$) that $x$ is bad. Notice that if $\hat Y_j(f) \ge 2$, then the chain $C_j$ must contain a bad vertex. Therefore, 
\[Q=\sum_j q_j \le \sum_j \pr(\text{some $x \in C_j$ is bad}) \le \sum_x r_x.\]

Set $s=p^{-1}\log(t^{1/2}(t-1)n)=(n/t)^{1/2}(\log (t^{1/2}(t-1)n))^{1/2}$. Say $y$ is a \textit{$k$-child} ($k \in \{1,\ldots,t-1\}$) of $x$ if $y \lessdot x$ and $y$ differs from $x$ in a coordinate that has the value $k$ in $x$. Given $f$, define a point $x$ is \textit{heavy} if there exists some $k \in \{1,\ldots,t-1\}$ for which $x$ has at least $s$ of $k$-children $y$ with $f(y)=1$, and \textit{light} otherwise. Say $x$ is \textit{occupied} if $f(x)=1$. 

Given $x \in [t]^n$, let $C_x \in \cC$ be the chain containing $x$, and $y=y(x)$ be $y \lessdot x$ and $y \in C_x$. Such $y$ exists by Proposition \ref{GK_extend} unless $x$ is the smallest in $C_x$. If $x$ is the smallest in $C_x$, then write $y(x)=\emptyset$. Note that 
\beq{bad.implication} \text{if $x$ is bad, then $x$ is high and occupied, $y(x) \ne \emptyset$, and $f(y)=1$.}\enq For convenience, denote by $Q_x$ the event that $x$ is high, occupied, and $y(x) \ne \emptyset$. Then by \eqref{bad.implication},

		\beq{rx.formula}\begin{split}
			r_x &=\pr(x \text{ heavy and bad})+\pr(x \text{ light and bad})\\
			&\leq \pr(x \text{ heavy and bad})+\pr(x\text{ light, $Q_x$})\cdot \pr(\bff(y)=1|x  \text{ light, $Q_x$}).\end{split}\enq
We will bound each term in the expression above.

\nin $\bullet$ \textit{Bounding $\pr(x \text{ heavy and bad})$.} Say $x$ is \textit{$k$-heavy} for $k \in \{1,2,\ldots,t-1\}$ if $x$ has at least $s$ of $k$-children $y$ with $f(y)=1$. In order for a $k$-heavy $x$ to be bad, each of the chains $\pi(C_i)$ containing the $k$-children $y$ with $f(y) = 1$ must have $\mathbf y_i=0$, which happens with probability at most $1-p$. Therefore,
\beq{bad given heavy}\begin{split} \pr(x \text{ heavy and bad}) & \le \sum_k \pr(\text{$x$ $k$-heavy and bad})\\
&\le \sum_k \pr(\text{$x$ bad$|$ $x$ $k$-heavy)} \le (t-1)(1-p)^s\leq (t-1)e^{-ps}=\frac{1}{t^{1/2}n}.\end{split}\enq

\smallskip

\nin $\bullet$ \textit{Bounding $\pr(\bff(y)=1|x  \text{ light, $Q_x$})$.} 
Suppose $y$ is a $k$-child of $x$ for some $k$, and let $N_k(x)$ be the set of $k$-children of $x$.  Consider $S_n$, the group of permutations on $[n]$. For $\pi \in S_n$, $\pi$ acts on $x \in [t]^n$ by $\pi(x)=(x_{\pi(1)},\ldots, x_{\pi(n)})$ and on $f \in \cF$ by
$\pi(f)=f(\pi(x)).$ Let $\text{Stab}(x)=\{\pi \in S_n: \pi(x)=x\},$ noticing that $\text{Stab}(x)$ acts transitively on $N_k(x)$. Therefore, for each $f \in \cF$, the fraction of elements $g$ in $\text{Orb}(f):=\{\pi(f):\pi \in \text{Stab}(x)\}$ with $g(y)=1$ is precisely $|f^{-1}(1) \cap N_k(x)|/|N_k(x)|$. Given that $x$ is high, light and occupied, $N_k(x) \ge n/(2t)$ and $|f^{-1}(1)\cap N_k(x)| \le s$. Finally, $\{\text{Orb}(f):f \in \cF\}$ partitions $\cF$. Therefore,
\beq{bad given light} \pr(\bff(y)=1|\text{$x$ light, $Q_x$})\le \frac{s}{\frac{n}{2t}}=\frac{2t^{1/2}(\log(t^{1/2}(t-1)n))^{1/2}}{n^{1/2}}.\enq

\smallskip

\nin $\bullet$ \textit{Bounding $\sum_x \pr(x\text{ light, $Q_x$})$.} Given $f$, write $R(f)$ for the set of points that are light and satisfy $Q_x$. Say $x$ is \textit{marginal} if $x$ is occupied, $y(x) \ne \emptyset$, and $f(y)=0$. Then
\[\begin{split} \pr(\text{$x$ non-marginal}|\text{$x$ light, $Q_x$})&\le \pr(\bff(y)=1|\text{$x$ light, $Q_x$}) \\
&\stackrel{\eqref{bad given light}}{\le} \frac{2t^{1/2}(\log(t^{1/2}(t-1)n))^{1/2}}{n^{1/2}},\end{split}\]
so,
\[\begin{split}\mathbb E[|\{x:\text{$x$ marginal, light, $Q_x$}\}|] \ge \left(1-\frac{2t^{1/2}(\log(t^{1/2}(t-1)n))^{1/2}}{n^{1/2}}\right)\mathbb E[|R(\bff)|].\end{split}\]
However, for any $f$, the number of marginal points is at most $N$ since each chain has at most one marginal point. Therefore, (using the fact that $1/(1-x) \le 1+2x$ for $x \in [0,1/2]$,)
\beq{sum.x.note}\sum_x\pr(\text{$x$ light, $Q_x$})=\mathbb E[|R(\bff)|] \le \left(1+\frac{4t^{1/2}(\log (t^{1/2}(t-1)n))^{1/2}}{ n^{1 / 2}}\right) N.\enq

Now, by combining \eqref{rx.formula}, \eqref{bad given heavy}, \eqref{bad given light}, and \eqref{sum.x.note},  we obtain that
 \[\begin{split} Q & \le \frac{(2+\eps')t^{1/2}(\log (t^{1/2}(t-1)n))^{1/2}}{n^{1/2}}N.
 \end{split}\]
\end{proof}

	\bibliographystyle{abbrv}
	\bibliography{references230402}

\appendix

\section{Proof of Proposition \ref{GK_extend}}

We extend the chain construction in \cite{Kleitman1976OnJCP} to $[t]^n$. This extension was also used in \cite{Tsai2019}. For each $x=(x_i) \in [t]^n=\{0,1,\ldots,t-1\}^n$, we define the corresponding "bracket configuration" as follows: consider $n$ consecutive blocks where each block consists of $t-1$ blanks. We fill in the blanks in the $i$th block from the left, first with $x_i$ right parentheses followed by $t-x_i-1$ left parentheses. Then define $B(x)$ to be the set of matched parentheses.

For example, if $x=(0,2,1,3,2,1)\in [4]^6$, then the corresponding bracket configuration and matched parentheses are

\begin{figure}[h!]
    \centering
    \includegraphics[width=8.5cm]{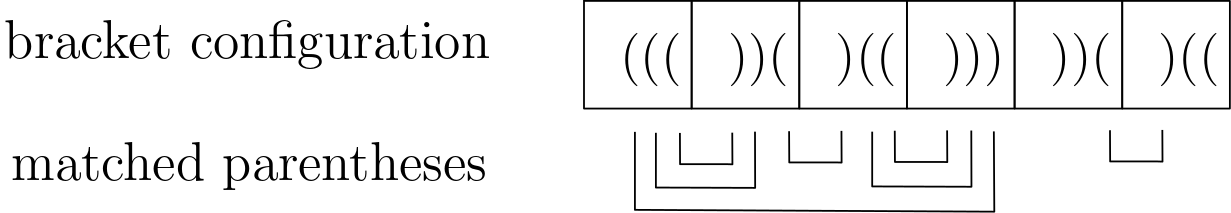}
\end{figure}

\nin For each $x \in [t]^n$, define the equivalent class $C_x=\{y \in [t]^n:B(y)=B(x)$\}. For example, for the above $x$, $C_x=\{(0,2,1,3,0,1), (0,2,1,3,1,1), (0,2,1,3,2,1), (0,2,1,3,2,2), (0,2,1,3,2,3)\}$.

 Then $\cC:=\{C_x:x \in [t]^n\}$ is a chain decomposition of $[t]^n$ of size $N$ (since each point in the middle layer is contained in exactly one chain in $\cC$). From the definition, it is easy to see that $\cC$ satisfies \eqref{WMA}.

\end{document}